\documentclass[12pt]{amsart}
\usepackage{amssymb,amsmath,amsthm}
\usepackage{enumerate}
\usepackage{mathrsfs}
\newtheorem{theorem}{Theorem}[section]

\newtheorem{corollary}[theorem]{Corollary}
\theoremstyle{definition}

\newtheorem{example}[theorem]{Example}
\newtheorem{remark}[theorem]{Remark}

\newcommand{\clh}{\mathcal{H}}

\newcommand{\raro}{\rightarrow}

\usepackage[utf8]{inputenc}
\newcommand\hlight[1]{\tikz[overlay, remember picture,baseline=-\the\dimexpr\fontdimen22\textfont2\relax]\node[rectangle,fill=white!50,rounded corners,fill opacity = 0.2,draw,thick,text opacity =1] {$#1$};}

\title{Orthogonality in Banach Spaces via projective tensor product}

\author[Dhara]{Kousik Dhara}
\address{Kousik Dhara, Indian Statistical Institute,
Statistics and Mathematics Unit, 8th Mile, Mysore Road,
Bangalore, 560 059, India.}
\email{kousik.dhara1@gmail.com}

\author[Rakshit]{Narayan Rakshit}
\address{Narayan Rakshit, Indian Statistical Institute,
Statistics and Mathematics Unit, 8th Mile, Mysore Road,
Bangalore, 560 059, India.}
\email{narayan753@gmail.com}

\author[Sarkar]{Jaydeb Sarkar}
\address{Jaydeb Sarkar, Indian Statistical Institute,
Statistics and Mathematics Unit, 8th Mile, Mysore Road,
Bangalore, 560 059, India.}
\email{jaydeb@gmail.com, jay@isibang.ac.in}

\author[Sensarma]{Aryaman Sensarma}
\address{Aryaman Sensarma, Indian Statistical Institute,
Statistics and Mathematics Unit, 8th Mile, Mysore Road,
Bangalore, 560 059, India.}
\email{aryamansensarma@gmail.com}


\numberwithin{equation}{section}
\begin{document}
\begin{abstract}
Let $X$ be a complex Banach space and $x,y\in X$. By definition, we say that $x$ is Birkhoff-James orthogonal to $y$ if $
\|x+\lambda y\|_{X} \geq \|x\|_{X}$ for all $\lambda \in \mathbb{C}$. We prove that $x$ is Birkhoff-James orthogonal to $y$ if and only if
there exists a semi-inner product $\varphi$ on $X$ such that $\|\varphi\| = 1$, $\varphi(x,x)=\|x\|^2$ and $\varphi(x,y)=0$. A similar result holds for $C^*$-algebras. A key point in our approach to orthogonality is the representations of bounded bilinear maps via projective tensor product spaces.
\end{abstract}
\maketitle
\section{Introduction}

This paper deals with the notion of Birkhoff-James orthogonality \cite{BIR,JA1,JA2} for vectors in Banach spaces via projective tensor product of Banach spaces. Let $X$ be a normed linear space over the scalar field $\mathbb{K}$, where $\mathbb{K}$ is either $\mathbb{R}$ or $\mathbb{C}$. Suppose $x, y \in X$. We say that $x$ is \textit{Birkhoff-James orthogonal} (or simply \textit{orthogonal}) to $y$ if
\[
\|x+\lambda y\|_{X} \geq \|x\|_{X},
\]
for all $\lambda \in \mathbb{K}$. We denote this by $x \perp_B y$. If $X$ is a Hilbert space, then this notion coincides with the usual orthogonality, that is, $x \perp_B y$ if and only if
\[
\langle x, y \rangle_{X} = 0.
\]
The objective of this note is to present an abstract characterization of orthogonality of vectors in Banach spaces over $\mathbb{K}$. We are motivated by Bhatia and \v{S}emrl's investigation \cite[Theorem 1.1 and Remark 3.1]{{BS}} of orthogonality in the setting of bounded linear operators on Hilbert spaces:

\begin{theorem}\label{thm-Bhatia and Semrl}(Bhatia and \v{S}emrl)
Let $S$ and $T$ be bounded operators on a Hilbert space $\clh$.

(i) Then $S \perp_B T$ if and only if there exists a sequence $\{h_n\}$ of unit vectors such that $\|S h_n\| \raro \|S\|$, and $\langle S h_n, T h_n \rangle \raro 0$ as $n \raro \infty$.

(ii) If we additionally assume that $\clh$ is a finite dimensional Hilbert space, then $S \perp_B T$ if and only if there exists a unit vector $h \in \clh$ such that
\[
\|S h\| = \|S\| \quad \mbox{and}\quad  \langle Sh, Th \rangle = 0.
\]
\end{theorem}

We treat the above Bhatia-\v{S}emrl result (more specifically, part (ii) of Theorem \ref{thm-Bhatia and Semrl}) as a paradigm and examine orthogonality of vectors in Banach spaces in terms of semi-inner products. Our main result, in the setting of Banach spaces, is the following:

\begin{theorem}\label{Complex Case}
Let $X$ be a Banach space, $x,y\in X$, and $x \neq 0$. Then $x \perp_B y$ if and only if there exists a semi-inner product $\varphi: X\oplus_{\infty} X \rightarrow \mathbb{K} $ such that $\|\varphi\| = 1$ and
\[
\varphi(x,x)=\|x\|^2 \quad \mbox{and}\quad  \varphi(x,y)=0.
\]
\end{theorem}

As an immediate consequence of the above, we also prove an orthogonality result in the setting of $C^*$-algebras (see Corollary \ref{C^*-algebra case}). Here, for a pair of Banach spaces $X$ and $Y$, we define $X\oplus_{\infty} Y$ to be the Banach space
\[
X \times Y = \{(x, y) : x \in X, y \in Y\},
\]
with the norm
\[
\|(x, y)\|_{\infty} = \max \{\|x\|_X, \|y\|_Y\},
\]
for all $x \in X$ and $y \in Y$. Also recall that a semi-inner product on a vector space $V$ is scalar-valued function $\varphi : V \times V \rightarrow \mathbb{K}$ such that for all $x,y,z\in V$ and $\alpha ,\beta \in \mathbb{K}$, we have
\begin{itemize}
\item[(a)] $\varphi(\alpha x+\beta y,z)=\alpha \varphi(x,z)+\beta \varphi(y,z)$,
\item[(b)] $\varphi(x,y)=\overline{\varphi(y,x)}$,
\item[(c)] $\varphi(x,x)\geq 0$.
\end{itemize}
It is easy to check that, if $\varphi$ is a semi-inner product, then $\varphi$ is an inner product if and only if $\varphi(x,x) = 0$ implies $x = 0$.

The main ingredients of our approach to the orthogonality problem are: (1)  Bhatia and \v{S}emrl's orthogonality of bounded linear operators on infinite-dimensional Hilbert spaces (part (i) of Theorem \ref{thm-Bhatia and Semrl} above), (2) projective tensor product techniques (see Theorem \ref{thm-known ptp} below), and (3) some standard Banach space techniques (like Banach-Alaoglu theorem).

Before we proceed with the main content of the paper, let us shortly review the existing literature on orthogonality. The notion of Birkhoff-James orthogonality is an active research area. In fact, Bhatia and \v{S}emrl's paper \cite{BS} on orthogonality of operators on Hilbert spaces has stimulated extensive research for the past two decades. For instance, orthogonality of a pair of compact operators acting on a reflexive Banach space to a normed linear space has been studied by Sain, Paul and Mal \cite{SPM}. Their investigation involves finer geometric Banach space techniques. Also see \cite{Sain} on orthogonality of linear operators on finite dimensional Banach spaces, \cite{Sain-Paul} on norm attainment and orthogonality, and \cite{TSS} on approximate orthogonality. More recent advances can be found, for instance, in the quickly growing literature
\cite{Arambai, TG, Chmieli, Komuro, Turn, Wojik} (also see the references therein).

Given Banach spaces $X$ and $Y$, we denote the Banach space of all bounded linear operators from $X$ to $Y$ by $\mathscr{B}(X,Y)$, and we let $X_1$ denote the closed unit ball in $X$. If $Y=X$, then we write $\mathscr{B}(X)$. A bilinear (sesquilinear) map $B: X\times Y \rightarrow Z$ is said to be bounded if there exists $M>0$ such that
\[
\|B(x,y)\|\leq M \quad \quad (x\in X_1, y\in Y_1).
\]
We denote the Banach space of all bounded bilinear (sesquilinear) maps from $X\times Y$ to $Z$ by $\text{Bil}(X\times Y, Z)$ ($\text{Ses}(X\times Y, Z)$). Here
\[
\|B\|=\sup \{\|B(x,y)\|: x \in X_1, \, y \in Y_1\},
\]
for all $B \in \text{Bil}(X\times Y, Z)$ ($B \in \text{Ses}(X\times Y, Z)$). As a tool for the proof of the main result, we use the notion of projective tensor product. The projective tensor product  $X\hat{\otimes}_\pi Y$ of Banach spaces $X$ and $Y$ is the completion of the algebraic tensor product $X \otimes Y$ under the projective norm
\[
\|u\|_\pi = \inf \Big\{ \sum _{i=1}^{n} \|x_i \| \|y_i\| :
u= \sum _{i=1}^{n}  x_i \otimes y_i \Big\}.
\]
Our key point is the following result \cite[Theorem 2.9]{RR} concerning representations of bounded bilinear maps via projective tensor product spaces.

\begin{theorem}\label{thm-known ptp}
Let $X, Y$ and $Z$ be Banach spaces, and let $B \in \text{Bil}(X \times Y,Z)$. Then there exists a unique $\tilde{B} \in \mathscr{B}(X\hat{\otimes}_\pi Y, Z)$ such that
\[
\tilde{B}(x\otimes y) = B(x, y) \quad \quad (x \in X, y \in Y).
\]
Moreover, the correspondence $B \longleftrightarrow \tilde{B}$ is an isometric isomorphism between $Bil(X \times Y,Z)$ and $\mathscr{B}(X\hat{\otimes}_\pi Y, Z)$.
\end{theorem}

We fix some more notation that we will use from now on. Given a vector space $X$ over $\mathbb{K}$, we denote by $\overline{X}$ the complex conjugate vector space of $X$. That is, $\overline{X} = X$ with the same additive group structures, but with the scalar multiplication $\star$ defined by
\begin{equation}\label{eqn-conjugate}
\alpha \star x=\overline{\alpha}x,
\end{equation}
for all $\alpha\in \mathbb{K}$ and $x \in \overline{X}$. Clearly, if $(X, \|.\|)$ is normed linear space over $\mathbb{K}$, then there is an anti-linear isometric isomorphism between $(X, \|.\|)$ and $(\overline{X}, \|.\|)$.  If $\mathcal{H}$ is a Hilbert space, then $\overline{\mathcal{H}}$ is identified with the dual (the space of continuous linear functionals) of $\mathcal{H}$ by Riesz representation theorem.

Note that the space $\mbox{Bil}(X \oplus_\infty \overline{ X},\mathbb{K})$ is isometrically isomorphic to the space $\mbox{Ses}(X \oplus_\infty X,\mathbb{K})$. The correspondence is given by
\begin{equation}\label{correspondence}
\psi(x,y)=	\varphi(x,y),
\end{equation}
where $\psi \in \mbox{Bil}(X \oplus_\infty X,\mathbb{K})$, $\varphi \in \mbox{Ses}(X \oplus_\infty X,\mathbb{K})$, $x \in X$ and $y \in Y$.

\section{Main Results}

We begin with the proof of Theorem \ref{Complex Case}. But before we do so, let us recall the classical Banach–Alaoglu theorem: Let $X$ be a Banach space over $\mathbb{K}$. Then the closed unit ball $(X^*)_1$ of the dual $X^*$ is compact with respect to the
$\mbox{weak}^*$ topology on $X^*$.

\smallskip

\noindent \textit{\textsf{Proof of Theorem \ref{Complex Case}:}}
	Suppose $x \perp_B y$. Note that the Banach space $X$ is isometrically isomorphic to a closed subspace of $C((X^*)_1)$, where $(X^*)_1$ endowed with the $\mbox{weak}^*$ topology is a compact set. Here the correspondence is given by the formula $X \ni u \mapsto \hat{u}$, where
\[
\hat{u}(f)=f(u) \quad \quad ( f \in (X^*)_1).
\]
Next we note that the commutative Banach algebra $C((X^*)_1)$ is isometrically isomorphic to a closed subspace of $\mathscr{B}(L^2(\mu))$ for some $\sigma $-finite measure $\mu$. Here the correspondence is given by the formula $C((X^*)_1) \ni g\mapsto M_g$, where $M_g: L^2(\mu) \rightarrow L^2(\mu)$ is the multiplication operator defined by
\[
M_g(h)=gh \quad \quad (h\in L^2(\mu)).
\]
Using the above identifications, we have $M_{\hat{x}}, M_{\hat{y}} \in \mathscr{B}(L^2(\mu))$ and
\[
M_{\hat{x}}\perp_B M_{\hat{y}}.
\]
By the infinite dimensional part of Theorem \ref{thm-Bhatia and Semrl}, there exists a sequence $\{h_n\} $ of unit vectors in $L^2(\mu)$ such that
\[
\|M_{\hat{x}}(h_n)\|_{L^2(\mu)} \rightarrow \|M_{\hat{x}}\|_{\mathscr{B}(L^2(\mu))},
\]
and
\[
\langle M_{\hat{x}}(h_n), M_{\hat{y}}(h_n) \rangle_{L^2(\mu)} \rightarrow 0 \text{ as } n \rightarrow \infty.
\]
For each $n \geq 1$, define $\psi_n : X\oplus_{\infty} \overline{X} \rightarrow \mathbb{K}$ by
\[
\psi_n (z,w)= \langle M_{\hat{z}}(h_n), M_{\hat{w}}(h_n) \rangle_{L^2(\mu)} \quad \quad (z \in X, w \in \overline{X}).
\]
Now we prove that $\psi_n$ is bilinear. Clearly, $\psi_n$ is linear in its first variable.
Suppose $\alpha_1, \alpha_2 \in \mathbb{K}$ and $z, w_1,w_2 \in X$. Then
\[
\begin{split}
\psi_n (z, (\alpha_1\star w_1 + \alpha_2 \star w_2)) & =
\langle M_{\hat{z}}(h_n), M_{\widehat{\alpha_1 \star w_1 + \alpha_2 \star w_2}}( h_n) \rangle
\\
& = \langle M_{\hat{z}}(h_n), M_{\widehat{\bar{\alpha_1} w_1 + \bar{\alpha_2}  w_2}}(h_n) \rangle
\\
&= \langle \hat{z}h_n, (\bar{\alpha_1} \hat{w_1} +\bar{\alpha_2 } \hat{w_2}) h_n \rangle
\\
& =\alpha_1 \langle \hat{z}h_n, \hat{w_1}h_n \rangle + \alpha_2  \langle \hat{z}h_n, \hat{w_2} h_n \rangle,
\\
\end{split}
\]
that is
\[
\psi_n\left(z, (\alpha_1\star w_1 + \alpha_2\star w_2)\right) = \alpha_1 \psi_n(z,w_1)+ \alpha_2 \psi_n (z,w_2),
\]
and hence $\psi_n$ is bilinear for all $n \geq 1$. To prove that $\psi_n$ is bounded, suppose $(z,w)\in X \oplus_{\infty} \overline{X}$ and $\|(z,w)\|_\infty \leq 1$.
% Recall that
%\[
%\|(z,w)\|_\infty := \max \{\|z\|, \|w\| \}.
%\]
Then
\begin{align*}
|\psi_n(z,w)| & \leq \|M_{\hat{z}}\|_{\mathscr{B}(L^2(\mu))} \; \| M_{\hat{w}} \|_{\mathscr{B}(L^2(\mu))} \; \|h_n \|^2_{L^2(\mu)}
\\
& = \|\hat{z}\|_{C((X^*)_1)} \|\hat{w}\|_{C((X^*)_1)}
\\
& = \|z\|_X \; \|w\|_X,
\end{align*}
which implies that $\|\psi_n\| \leq 1$ for all $n$. Consequently
\[
\{\psi_n\}_{n\geq 1} \subseteq \text{Bil}(X\oplus_{\infty} \overline{X}, \mathbb{K}).
\]
Note, by Theorem \ref{thm-known ptp}, that
\[
\text{Bil}(X\oplus_{\infty}\overline{ X}, \mathbb{K}) \simeq \mathscr{B}(X\hat{\otimes}_\pi \overline{X}, \mathbb{K}).
\]
Here the correspondence is given by $\text{Bil}(X\oplus_{\infty}\overline{ X}, \mathbb{K}) \ni \psi \mapsto \tilde{\psi}$, where
\[
\tilde{\psi} (z\otimes w) = \psi(z,w) \quad \quad (z,w\in X).
\]
Since $\|\psi_n\| \leq 1$, we have $\|\tilde{\psi_n} \| \leq 1$ for each $n$. By applying the Banach-Alaoglu theorem, one can find a subsequence $\{\tilde{\psi}_{n_k}\}$ of $\{\tilde{\psi_n}\}$  and a map $\tilde{\psi}\in \mathscr{B}(X\hat{\otimes}_\pi \overline{X}, \mathbb{K}) $ with $\|\tilde{\psi}\| \leq 1$ such that
\[
\tilde{\psi}_{n_k} \mathop\rightarrow^{w^*} \tilde{\psi}.
\]
This yields $\psi_{n_k} \rightarrow \psi$ in the pointwise topology, that is
\[
\psi_{n_k}(z,w) \rightarrow \psi(z,w) \quad \quad (z,w\in X).
\]
By isometry, $\|\psi\|\leq 1$. Observe that
\[
\psi_{n_k}(x,x)  =  \langle M_{\hat{x}  }(h_{n_k}), M_{\hat{x}}(h_{n_k}) \rangle_{L^2(\mu)} = \|M_{\hat{x}  }(h_{n_k})\|^2_{L^2(\mu)}.
\]
But since
\[
\|M_{\hat{x}  }(h_{n_k})\|^2_{L^2(\mu)} \rightarrow \|M_{\hat{x}}\|^2_{\mathscr{B}(L^2(\mu))} = \|\hat{x}\|^2,
\]
it follows that
\[
\psi_{n_k}(x,x) \rightarrow \|x\|^2,
\]
for all $x \in X$. On the other hand, $\psi_{n_k}(x,x)\rightarrow \psi(x,x)$. Hence
\[
\psi(x,x)=\|x\|^2 \quad \quad (x \in X).
\]
Since
\[
\psi_{n_k}(x,y)=  \langle M_{\hat{x}  }(h_{n_k}), M_{\hat{x}}(h_{n_k}) \rangle_{L^2(\mu)} \rightarrow 0,
\]
we have $\psi(x,y)=0$. Note that if $\|\psi\|<1$, then $\psi(x,x)<\|x\|^{2}$. Hence $\|\psi\|=1$. Now observe that the identification in \eqref{correspondence} provides a sesquilinear map
\[
\varphi \in \text{Ses}(X\oplus_\infty X, \mathbb{K}),
\]
with
\[
\varphi(x,y)=\psi(x,y) \quad \quad (x,y\in X).
\]
This gives $\varphi(z,z)\geq 0$ for all $z \in X$, and $\varphi(z,w)=\overline{\varphi(w,z)}$ for all $z,w \in X$.
Hence $\varphi$ is a semi-inner product on $X$ such that $\|\varphi\|=1$, $\varphi(x,x)=\|x\|^{2}$ and $\varphi(x,y)=0$.

\noindent Conversely, suppose that such a $\varphi$ exists. Then for all $\lambda\in \mathbb{K}$,
\[
\|x\|^2 = |\varphi(x,x+\lambda y)| \leq \|\varphi\| \|x\| \|x+\lambda y\|,
\]
which implies that $\|x\| \leq \|x+\lambda y\|$, that is, $x \perp_B y$. This completes the proof of the theorem.

\begin{remark}
It is worth pointing out that besides Theorem \ref{thm-Bhatia and Semrl} and Theorem \ref{thm-known ptp}, the inclusion
\[
X \hookrightarrow C((X^*)_1) \hookrightarrow \mathscr{B}(L^2(\mu)),
\]
also plays an important role in our proof.
\end{remark}

Now we turn to orthogonality of elements in $C^*$-algebras. The following result follows from Theorem \ref{Complex Case}.

\begin{corollary}\label{C^*-algebra case}
	Let $\mathcal{A}$ be a $C^*$-algebra and $a,b \in \mathcal{A}$ and $a\neq 0$. Then $a\perp_B b$ if and only if there exists a bilinear map $\psi: \mathcal{A}\oplus_{\infty} \mathcal{A} \rightarrow \mathbb{C}$ such that $\|\psi\|= 1$, $\psi(a,a^*)=\|a\|^2$ and $\psi(a,b^*)=0$.
\end{corollary}

\begin{proof}
Suppose $a\perp_B b$. Then by Theorem \ref{Complex Case}, there exists a semi-inner product $\varphi: \mathcal{A}\oplus_{\infty} \mathcal{A} \rightarrow \mathbb{C}$ such that $\|\varphi\|= 1$, $\varphi(a,a)=\|a\|^2$ and $\varphi(a,b)=0$. Define $\psi: \mathcal{A}\oplus_{\infty} \mathcal{A} \rightarrow \mathbb{C} $ by
	$$ \psi(a,b)=\varphi(a,b^*).$$
	Then $\psi$ is the required map. To prove the converse, suppose such a $\psi$ exists. Then
\begin{align*}
\|a\|^2 &=|\psi(a, (a+\lambda b)^*|\\
& \leq \| \psi\| \|a\| \|(a+\lambda b)^* \| \\
& = \|a\| \|a+\lambda b\|,
\end{align*}
that is, $\|(a+\lambda b)\| \geq \|a\|$ for all $\lambda \in \mathbb{C}$

\end{proof}

 It is worth pointing out that Theorem \ref{Complex Case} is also applicable to finite dimensional Banach spaces. Clearly, Theorem \ref{Complex Case} and Corollary \ref{C^*-algebra case} are analogous to Bhatia and \v{S}emrl classifications (see part (ii) of Theorem \ref{thm-Bhatia and Semrl}) of orthogonality of matrices on finite dimensional Hilbert spaces. On the other hand, our results uses the Bhatia and \v{S}emrl classifications of orthogonality in the setting of infinite dimensional Hilbert spaces (see part (i) of Theorem \ref{thm-Bhatia and Semrl}). In addition, it is not completely clear if our results recovers the Bhatia and \v{S}emrl classifications of orthogonality in the setting of finite matrices.

All in all, on one hand our results are valid for general Banach spaces and rather abstract, and on the other hand our approach is intimately related to the delicate structure of projective tensor product of Banach spaces (see Chapter 2 in \cite{RR}). We  also believe that our approach to orthogonality via projective tensor product is of independent interest and may have other applications.

\vspace{0.2in}

\noindent\textbf{Acknowledgement:}
The research of the third named author is supported in part by NBHM grant NBHM/R.P.64/2014, and the Mathematical Research Impact Centric Support (MATRICS) grant, File No: MTR/2017/000522 and Core Research Grant, File No: CRG/2019/000908, by the Science and Engineering Research Board (SERB), Department of Science \& Technology (DST), Government of India. The research of the fourth author is supported  by the NBHM postdoctoral fellowship, File No:  0204/3/2020/R$\&$D-II/2445.

\end{document}